\documentclass[12pt,oneside]{article}
\usepackage[all]{xy}
\usepackage{amsmath,amssymb,amsfonts,amsthm,mathrsfs,dsfont,mathtools}
\usepackage[english]{babel}
\usepackage{graphicx,stackrel}
\usepackage{multirow,rotating,paralist}
\usepackage[toc,page,header]{appendix}
\DeclareMathAlphabet{\mathpzc}{OT1}{pzc}{m}{it}
\usepackage{minitoc}
\usepackage{extarrows}
\usepackage{hyperref}
\usepackage{csquotes}

 \small {\par\vskip8pt minus3pt\rm}
\newcounter{item}[section]
\newcounter{kirshr}
\newcounter{kirsha}
\newcounter{kirshb}

\newtheorem{theorem}{Theorem}[section]

\newtheorem{lemma}[theorem]{Lemma}

\newtheorem{proposition}[theorem]{Proposition}

\newtheorem{remark}[theorem]{Remark}

\newtheorem{definition}[theorem]{Definition}

\newcommand\undersym[2]{\raisebox{-6pt}{\tiny$#2$}{\kern-5pt}\mbox{$#1$}}
\newcommand\overcirc[1]{\raisebox{10pt}{\tiny$\circ$}{\kern-7pt}\mbox{$#1$}}

\title{\Large{Classification of 1-Weierstrass points on Kuribayashi quartics,\hspace*{-.12cm} I (with two parameters)}\footnote{This paper was presented in \textquotedblleft International
Conference on Mathematics, Trends, Development
(ICMTD12)\textquotedblright\,\, held at Cairo, Egypt, December
27-29, 2012.}}
\date{}
\author{{\small Eslam E. Badr$^{1}$ and Mohammed A. Saleem$^{2}$}}

\begin{document}
\bibliographystyle{plain}
\maketitle                     
\vspace*{-.95cm}

\begin{center}
${^1}$Department of Mathematics, Faculty of Science, \\[-0.15cm]
Cairo University, Giza, Egypt \\[-0.15cm]
\vspace*{.3cm}
${^2}$Department of Mathematics, Faculty of Science,\\[-0.15cm]
Sohag University, Sohag, Egypt \\[-0.15cm]
\end{center}

\begin{center}
$
\begin{array}{ll}
\text{Emails:}&\text{eslam@sci.cu.edu.eg, eslam60145@yahoo.com} \\
  &\text{abuelhassan@yahoo.com}
\end{array}
$
\end{center}

\maketitle \bigskip

\begin{abstract}
In this paper, we classify the 1-Weierstrass points of the
Kuribayashi quartic curves with two parameters $a$ and $b$ defined
by the equation
\[
C_{a,b}: x^{4}+y^{4}+z^{4}+ax^{2}y^{2}+b(x^{2}+y^{2})z^{2}=0,
\]
such that $(a^2-4)(b^2-4)(a^2-1)(b^{2}-a-2)\neq0$. Furthermore, the geometry of these points is investigated.
\end{abstract}
\newpage

\section{Introduction}
Let $C_{a}$ be the smooth plane quartic curves defined by the
equation
\[
C_{a}:x^{4}+y^{4}+z^{4}+a(x^{2}y^{2}+x^{2}z^{2}+y^{2}z^{2})=0,\quad\quad
a\neq-1,\pm2.%
\]
These types of quartic curves are called \emph{Kuribayashi quartics
with one parameter. }Weierstrass points and automorphism groups of
Riemann surfaces of genus $3$ are studied in \cite{pa11, pa15}. In
particular, Weierstrass points and automorphism groups of $C_a$ are
studied in \cite{pa12}. Kuribayashi and his students used \emph{the
Wronskian method} to classify the number of the 1-Weierstrass points
of $C_a$. Alwaleed \cite{pa3, pa4} used the \emph{the Wronskian
method} together with the $S_4$ action on $C_a$ to classify the
number and investigate the geometry of the 2-Weierstrass points of
this family.

 Let $C_{a,b}$ be the smooth plane quartic curves defined by the equation
\[
C_{a,b}:x^{4}+y^{4}+z^{4}+ax^{2}y^{2}+b(x^{2}+y^{2})z^{2}=0,
\]
where $a$ and $b$ are two parameters such that
$(a^2-4)(b^2-4)(a^2-1)(b^{2}-a-2)\neq0.$ We call these quartic
curves \emph{Kuribayashi quartics with two parameter. }Hayakawa
\cite{pa5} investigated the conditions under which the number of the
Weierstrass points of $C_{a,b}$ is $12$ or $16$.

 In this paper, we use finite group actions on the Riemann surfaces $C_{a,b}$ to investigate the number of the 1-Weierstrass points of this family. Furthermore, we study the geometry of these points.

The present paper is organized in the following manner. In section
$1,$ we present some preliminaries concerning the basic concepts
that will be used throughout the work \cite{pa8, pa13}. In section
$2,$ we establish our main results \emph{(Theorems 2.6,\,\,\,2.14)} that
concern with the classification of the number of the 1-Weierstrass
points of the quartic curves $C_{a,b}$ together with their geometry.
In section $3$, we illustrate, through examples, the cases mentioned
in the main results. Finally, we conclude the paper with some
remarks, comments and related problems.

\section{Preliminaries}

\subsection{q-Weierstrass points}

Let $C$ be a smooth projective plane curve of genus $g\geq2$ and let $D$ be a
divisor on $C$ with $dim|D|=r\geq0$. We denote by $L(D)$ the $\mathbb{C}%
$-vector space of meromorphic functions $f$ such that $div(f)+D\geq0$ and
by $l(D)$ the dimension of $L(D)$ over $\mathbb{C}$. Then, the
notion of $D$-Weierstrass points \cite{pa13} can be defined in the following way:
\begin{definition}
Let $p\in C.$ If $n$ is a positive integer such that
\[
l\big(D-(n-1).p\big)>l\big(D-n.p\big),
\]
we call the integer $n$ a $D$-gap number at $p$.
\end{definition}

\begin{lemma}
Let $p\in C,$ then there are exactly $r+1$ $D$-gap numbers $\{n_{1},n_{2},...,n_{r+1}\}$ such that $n_{1}<n_{2}<...<n_{r+1}$. The sequence $\{n_{1},n_{2},...,n_{r+1}\}$ is called the $D$-gap sequence at $p$.
\end{lemma}

\begin{definition}
The integer
$\omega_{D}(p):=\sum_{i=1}^{r+1}(n_{i}-1)$ is called $D$-weight at $p$. If
$\omega_{D}(p)>0$, we call the point $p$ a $D$-Weierstrass point on $C$. In
particular, for the canonical divisor $K$, the $qK$-Weierstrass points $(q\geq1)$
are called $q$-Weierstrass points and the $qK$-weight is called $q$-weight, denoted by $\omega^{(q)}(p)$.
\end{definition}

\begin{definition}
\em{\cite{pa4}} A point $p$ on a smooth plane curve $C$ is said to
be a flex point if the tangent line $L_p$ meets $C$ at $p$ with
contact order $I_{p}(C,L_{p})$ at least three. We say that $p$ is
$i$-flex, if $I_{p}(C,L_{p})-2=i$. The positive integer $i$ is
called the flex order of $p.$
\end{definition}

\begin{lemma}
\em{\cite{pa8}} Let $C: F(x,y,z)=0$ be a smooth projective plane
curve. A point $p$ on $C$  is a flex point if, and only if,
$H_F(p)=0,$ where $H_F$ is the Hessian curve of $C$ defined by
\[H_F:=det\left(
             \begin{array}{ccc}
             F_{xx} & F_{xy} & F_{xz} \\
             F_{yx} & F_{yy} & F_{yz} \\
              F_{zx} & F_{zy} & F_{zz}
                    \end{array}
                  \right)
.\]

\end{lemma}

\begin{lemma}
\em{\cite{pa14}} Let $C$ be a smooth projective plane quartic curve, the $1$-Weierstrass points on $C$ are nothing but flexes and divided into two types: ordinary flex and hyperflex points. Moreover, we have%

\begin{center}
    \begin{tabular}
[c]{|c|c|c|}\hline
$\omega^{(1)}(p)$ & $1$-gap Sequence & Geometry\\\hline
$1$ & $\{1,2,4\}$ & ordinary flex\\\hline
$2$ & $\{1,2,5\}$ & hyperflex\\\hline
\end{tabular}
\end{center}

\emph{\ }
\end{lemma}

\begin{lemma}
\em{\cite{pa7, pa13}} Let $C$ be a smooth projective plane curve of
genus $g$. The number of $q$-Weierstrass points $N^{(q)}(C),$
counted with their $q$-weights, is given by
\[
N^{(q)}(C)=\left\{
\begin{array}
[c]{lr}%
g(g^{2}-1),\,\,\,\,\,\,\,\,\,\,\,\,\,\,\,\,\,\,\,\,\,\,\,\,\,\,\,\,\,\,\,\,\,\,\,\,\,\,if\, q=1 & \\\\
(2q-1)^{2}(g-1)^{2}g,\,\,\,\,\,\,\,\,\,\,\,\,\,\,\,if\, q\geq2. &
\end{array}
\right.
\]
In particular, for smooth projective plane quartics $($i.e.\,\,$g=3)$, the number
of $1$-Weierstrass points is $24$ counted with their weights.
\end{lemma}

Let $W^{(q)}(C)$ be the set of $q$-Weierstrass points on $C$ and $G^{(q)}(p)$ the $q$-gap sequence at the point $p\in C$.

\begin{lemma} \em{\cite{pa3}}
Let $\tau$ be an automorphism on $C$, then we have
\[
\tau\big(W^{(q)}(C)\big)=W^{(q)}(C)\quad and\quad G^{(q)}(\tau(p))=G^{(q)}(p).
\]

\end{lemma}

\bigskip
\subsection{Group Action on Riemann Surfaces}
\vspace*{-.9cm}$\phantom{GroupActiononRiemannSurfaceaaaaaaaaaaaaaaaaaaaa}$[13]
\begin{definition}
An action of a finite group $G$ on a Riemann surface $C$ is a map
\[
\cdot:G\times C\longrightarrow C:(g,p)\longmapsto g\cdot p
\]
such that: $(gh)\cdot p=g\cdot(h\cdot p)$\,\,and\,\,$e\cdot p=p,$\,\,for all
$g,h\in G$ and $p\in C$. Here $e$ denotes the identity element of $G.$
\end{definition}

\begin{definition}
The orbit of a point $p\in C$ is the set\,\, $Orb_{G}(p):=\{g\cdot p:g\in G\}.$
\end{definition}

\begin{definition}
The stabilizer of a point $p\in C$ is the subgroup
\[
G_{p}:=\{g\in G:g\cdot p=p\}.
\]
It is often called the isotropy subgroup of $p.$
\end{definition}

\begin{remark}
\em{It is well known that $G_{p}$ is cyclic and points in the same
orbit have conjugate stabilizers; Indeed, $G_{g\cdot
p}=gG_{p}g^{-1}.$ Moreover
$|Orb_{G}(p)|\,|G_{p}|=|G|.$}\vspace*{.3cm}
\end{remark}
\noindent\textbf{Notation.}\,[1] The set of points $p\in C$ such
that $|G_{p}|>1$ is denoted by $X(C).$ Also
\[
X_{i}(C):=\{p\in C:|G_{p}|=i\}.
\]


\section{Main Results}

In what follows, we investigate the number of the 1-Weierstrass points of \emph{Kuribayashi quartic curves with two parameter} $a$ and $b$ defined by the equation
\[
C_{a,b}:x^{4}+y^{4}+z^{4}+ax^{2}y^{2}+b(x^{2}+y^{2})z^{2},
\]
such that $(a^2-4)(b^2-4)(a^2-1)(b^{2}-a-2)\neq0.$
Moreover, the geometry of these points is studied. We have two cases, either $b=0$ or $b\neq0$. In the following we treat each of these cases.

\subsection{Case $b=0$}

In this case, we consider the one parameter family of smooth plane quartics
$C_{a,0}$ defined by%
\[
C_{a,0}:\quad x^{4}+y^{4}+z^{4}+ax^{2}y^{2},\quad\quad (a^2-1)(a^2-4)\neq0.
\]
A group action of order $16$ on $C_{a,0}$ can be defined as follows.
Let $G$ be the projective transformation group generated by the
three elements $\sigma,\,\tau,\,\rho$ of orders $2,4,4$
respectively, where
\[
\sigma:=\left(
\begin{array}
[c]{ccc}%
-1 & 0 & 0\\
0 & 1 & 0\\
0 & 0 & 1
\end{array}
\right),\quad\tau:=\left(
\begin{array}
[c]{ccc}%
i & 0 & 0\\
0 & -i & 0\\
0 & 0 & 1
\end{array}
\right),\quad\rho:=\left(
\begin{array}
[c]{ccc}%
0 & -1 & 0\\
1 & 0 & 0\\
0 & 0 & 1
\end{array}
\right).
\]
It has been shown by Francesc \cite{pa23} that, $G\cong
C_4\circledcirc \left(C_2\times C_2\right)$, where $C_m$ denotes the
cyclic group of order $m$ and $C_4\circledcirc \left(C_2\times
C_2\right)$ denotes the group $(16, 13)$ in GAP library of small
groups which is a group in $Ext^1\left(C_2\times C_2, C_4\right)$.
Now, computing the fixed points of the automorphisms of $G$ on
$C_{a,0}$ and their corresponding orbits gives rise to the following
result.

\begin{lemma}
For the quartics $C_{a,0},$ we have:
\[X(C_{a,0})=X_{2}(C_{a,0})\cup X_{4}(C_{a,0}).\]
Moreover, if $a=3$, then
\[X(C_{a,0})=Orb_{G}[\beta:1:0]\cup Orb_{G}[\alpha:\alpha:1]\cup Orb_{G}[0:\delta:1]\cup Orb_{G}[i:1:1],\]
otherwise; \[X(C_{a,0})=Orb_{G}[\beta:1:0]\cup Orb_{G}[\alpha:\alpha:1]\cup Orb_{G}[0:\delta:1],\]
where $\beta$ is a root of the equation $x^{4}+ax^{2}+1=0,$\,\,  $\alpha$ is a root of the equation $(a+2)x^{4}+1=0$\,\,and $\delta$ is a root of the equation $x^{4}+1=0.$
\end{lemma}

\begin{proposition}
For the quartics $C_{a,0},$ we have:
\newline\newline
$X(C_{a,0})\cap W_{1}(C_{a,0})=Orb_{G}[\beta
:1:0]\cup\left\{
\begin{array}
[c]{lr}%
Orb_{G}[0:\delta:1],\,\,\,\,\,\,\,\,\,\,\,\,\,\,\,\,\,\,\,\,\,\, if\,\,a=0 & \\\\
Orb_{G}[\alpha:\alpha:1],\,\,\,\,\,\,\,\,\,\,\,\,\,\,\,\,\,\,\,\,if\,\,a=6 & \\\\
\phi,\,\,\,\,\,\,\,\,\,\,\,\,\,\,\,\,\,\,\,\,\,\,\,\,\,\,\,\,\,\,\,\,\,\,\,\,\,\,\,\,\,\,\,\,\,\,\,\,\,\,\,\,\,\,
otherwise &
\end{array}
\right.  $
\end{proposition}

\begin{proof}
The Hessian curve $H_F$ of $C_{a,0}$ is given by
\[H_F(x,y,z)=-144\left(\left(-12+a^2\right)x^2y^2-2a\left(x^4+y^4\right)\right)z^2,\]
consequently, $H_F(\beta,1,0)=0,$ that is \[Orb_G[\beta:1:0]\subset W_1(C_{a,0}).\] Also, the resultant of $F(x,y,1)$ and $H_F(x,y,1)$ with respect to $y$ is given by \[Res\left(H_F(x,y,1),F(x,y,1);y\right)=constant\left(h(x)\right)^2,\]
where
\[h(x):=4a^2+\left(144-48a^2+3a^4\right)x^4+\left(144-72a^2+9a^4\right)x^8.\]
Hence \[Orb_G[0:\delta:1]\subset W_1(C_{a,0})\,\,\,\, \text{if, and
only if,}\,\,\, a=0.\] On the other
hand,\[H_F[\alpha:\alpha:1]=-144(a-6)(a+2)\alpha^4,\] consequently,
\[Orb_G[\alpha:\alpha:1]\subset W_1(C_{a,0})\,\,\,\text{if, and only
if,}\,\,\,a=6.\] Moreover,
\[H_F[i:1:1]=144\left((a+6)(a-2)\right)\neq0,\] whenever $a=3,$ that
is,\,\, $Orb_G[i:1:1]\cap W_1(C_{a,0})=\phi.$

This completes the proof.
\end{proof}

\begin{lemma}
If $[\gamma:\zeta:1]\notin X(C_{a,0})$, then
\[
Orb_{G}[\gamma:\zeta:1]=\{[\pm\gamma:\pm\zeta:1],[\pm i\gamma:\pm
i\zeta:1],[\pm\zeta:\pm\gamma:1],[\pm i\zeta:\pm i\gamma:1]\},
\]
and
\[
|Orb_{G}[\gamma:\zeta:1]|=16.
\]

\end{lemma}

\begin{lemma}
If $a\neq0,6$, then there exists an ordinary flex point $[\gamma
:\zeta:1]\notin X(C_{a,0})$, such that
\[
W_{1}(C_{a,0})=Orb_{G}[\beta:1:0]\cup Orb_{G}[\gamma:\zeta:1].
\]
\end{lemma}

\begin{lemma}
If $X(C_{a,0})\cap W_{1}(C_{a,b})\neq\phi$, then the intersection
points are necessarily hyperflex points.
\end{lemma}
\noindent\textbf{Notation.} Let $O_{r}$ \ be the orbits
classification, where $O$ denotes the number of orbits and $r$ the
number of points in these orbits. For example $2_{4}$ means: two
orbits each of 4 points. Now, our main result for the case $b=0$ is
the following.
\begin{theorem}
For the quartics $C_{a,0},$ the number of the 1-Weierstrass points of the quartics $C_{a,0}$ together with their geometry are given in the following table.

\begin{center}
  \textit{Number and Orbit Classification on $C_{a,0}$}%
\end{center}
\[%
\begin{tabular}
[c]{|c|c|c|}\hline
& \emph{Ordinary flexes} & \emph{Hyperflexes}\\\hline
$\ $ & $\ $ & $\mathbf{12}$\\
$a=0,6$ & $\mathbf{0}$ & $1_{4}$\\
&  & $1_{8}$\\\hline
\em{Otherwise} & $\mathbf{16}$ & $\mathbf{4}$\\
& $1_{16}$ & $1_{4}$\\\hline
\end{tabular}
\]
\\
where the boldface numbers denote the number of points.
\end{theorem}

\begin{proof}
Recall that the number of the 1-Weierstrass points is 24 counted
with their weights. Now, if $a=0$ or $a=6,$\, then by
\emph{Proposition 2.2} and \emph{Lemma 2.5}, the intersection
$X(C_{a,0})\cap W_{1}(C_{a,0})$ consists of $1_4$ and $1_8$ of
hyperflex points, which proves the case. If $a\neq0$ and $a\neq6$,
then again by \emph{Proposition 2.2} and \emph{Lemma 2.5},
$X(C_{a,0})\cap W_{1}(C_{a,0})$ consists of $1_4$ of hyperflex
points. Then it follows, by \emph{Lemma 2.3}, that there is $1_{16}$
of ordinary flex points.
\end{proof}

\subsection{Case $b\neq0$}

A group action of order $8$ on $C_{a,b}$ can be defined as follows.
Let $G_{1}\cong D_4$ be the projective transformation group
generated by the two elements $\sigma_{1}$ and $\tau_{1}$ of orders
$2$ and $4$ respectively, where
\[
\sigma_{1}:=\left(
\begin{array}
[c]{ccc}%
0 & 1 & 0\\
1 & 0 & 0\\
0 & 0 & 1
\end{array}
\right),\quad\tau_{1}:=\left(
\begin{array}
[c]{ccc}%
0 & 1 & 0\\
-1 & 0 & 0\\
0 & 0 & 1
\end{array}
\right).
\]
Now, computing the fixed points of the automorphisms of $G_1$ on
$C_{a,b}$ and their corresponding orbits gives rise to the following
result.

\begin{lemma}
For the quartics $C_{a,b}$ such that $b\neq0,$ we have:
\begin{eqnarray*}
X(C_{a,b})&=&Orb_{G_{1}}[\alpha_{1}:\alpha_{1}:1]\cup
Orb_{G_{1}}[\alpha_{3}:\alpha_{3}:1]\cup Orb_{G_{1}}[\beta:0:1]\cup\\
&&\cup\,Orb_{G_{1}}[\dfrac{1}{\beta}:0:1]\cup
Orb_{G_{1}}[\delta:1:0],
\end{eqnarray*}
where $\alpha_{1}$ and $\alpha_{3}$ are two distinct roots of the
equation $(a+2)x^{4}+2bx^{2}+1=0$ such that
$\alpha_{1}\neq\pm\alpha_{3},$ $\beta$ is a root of the equation
$x^{4}+bx^{2}+1=0$ and $\delta$ is a root of the equation
$x^{4}+ax^{2}+1=0.$
\end{lemma}

\begin{remark}
\em{Each of the above orbits satisfies $|Orb_{G_{1}}P|=4.$ This
means that
\[
\,\,\,\,\,\,Orb_{G_{1}}\left[\alpha_{1,3}:\alpha_{1,3}:1\right]=\left\{\left[\pm\alpha_{1,3}:\pm\alpha
_{1,3}:1\right]\right\},
\]%
\[
\,\,\,\,\,Orb_{G_{1}}\left[\beta:0:1\right]=\left\{\left[\pm\beta:0:1\right],\left[0:\pm\beta:1\right]\right\},
\]%
\[
\,\,\,\,\,\,Orb_{G_{1}}[\dfrac{1}{\beta}:0:1]=\{[\pm\dfrac{1}{\beta}:0:1],[0:\pm\dfrac
{1}{\beta}:1]\},
\]%
\[
Orb_{G_{1}}\left[\delta:1:0\right]=\left\{\left[\pm\delta:1:0\right],\left[1:\pm\delta:0\right]\right\}.
\]
Moreover, if\,$[\zeta:\varepsilon:1]\notin X(C_{a,b})$,\, then $|Orb_{G_{1}}[\zeta:\varepsilon:1]|=8.$}
\end{remark}
\medskip
Now, the intersections $X(C_{a,b})\cap W_{1}(C_{a,b})$ are given by the following three lemmas.
\begin{lemma}
Let \[A:=Orb_{G_{1}}[\beta:0:1]\cup Orb_{G_{1}}[\dfrac{1}{\beta}:0:1],\] then for the quartics $C_{a,b}$ such that $b\neq0$, we have:
\[A\cap W_1(C_{a,b})\neq\phi\text{ }\\
\text{
\ \    }%
\mathrm{if,\,\,and\,\, only\,\, if,}\text{ }\,\,\, a^2+b^2-ab^2=0.\]
Moreover, \[A\cap W_{1}(C_{a,b})= \left\{
\begin{array}
[c]{l}%
Orb_{G_{1}}[\beta:0:1], \text{
\ \ \ \ \ \ \ \ \ \ \ \ }%
\text{\,if\,}\text{ }a=\dfrac{1}{2}(b^{2}-b\sqrt{b^{2}-4})\\\\
Orb_{G_{1}}[\dfrac{1}{\beta}:0:1], \text{
\ \ \ \ \ \ \ \  \ \ \ \  }%
if\text{ }a=\dfrac{1}{2}(b^{2}+b\sqrt{b^{2}-4})\\\\
\phi,\text{ \ \ \ \ \ \ \ \ \ \ \ \ \ \ \ \ \ \ \ \ \ \ \ \ \ \ \ \ \ }otherwise,%
\end{array}
\right.
\]
where $\beta=\dfrac{\sqrt{2}}{\sqrt{-b+\sqrt{b^{2}-4}}}.$

\end{lemma}

\begin{proof}
The Hessian curve $H_F(x,0,1)$ is given by
\begin{equation*}
H_F(x,0,1)=\left(2b+12x^2\right)\left(2b+2ax^2\right)\left(12+2bx^2\right)+4bx\left(-8b^2x-8abx^3\right),
\end{equation*}
and so the resultant of $H_F(x,0,1)$ and $F(x,0,1)$ with respect to $x$ is
\begin{equation*}
Res\left(H_F(x,0,1),F(x,0,1);x\right)=const.\left(g(a,b)\right)^2,
\end{equation*}
where
\begin{equation*}
g(a,b)=\left(b^2-4\right)^2\left(a^2+b^2-ab^2\right).
\end{equation*}
From the last equation, we have
\begin{equation*}
A\cap W_1(C_{a,b})\neq\phi \quad \text{\,if,\,\,and\,\,only\,\,if,\,}\quad a^2+b^2-ab^2=0.
\end{equation*}
Moreover, substituting $a=\dfrac{1}{2}(b^2+b\sqrt{b^2-4})$ into $%
H_F(x,0,1)=0=F(x,0,1)$ yields the system
\begin{equation*}
x^4+x^2+1=0,
\end{equation*}
\begin{equation*}
\left(2+bx^2+\sqrt{-4+b^2}x^2\right)\left(-12x^2+b^2x^2-2b(1+x^4)\right)=0,
\end{equation*}
which has the two solutions $\pm\dfrac{\sqrt{2}}{\sqrt{-b-\sqrt{b^2-4}}}.$
$\newline$
Also, substituting $a=\dfrac{1}{2}(b^2-b\sqrt{b^2-4})$ into $%
H_F(x,0,1)=0=F(x,0,1)$ yields the system
\begin{equation*}
x^4+x^2+1=0,
\end{equation*}
\begin{equation*}
\left(2+bx^2-\sqrt{-4+b^2}x^2\right)\left(-12x^2+b^2x^2-2b(1+x^4)\right)=0,
\end{equation*}
which again has the two roots
$\pm\dfrac{\sqrt{2}}{\sqrt{-b+\sqrt{b^2-4}}}$.
\end{proof}
\bigskip

\begin{lemma}
For the quartics $C_{a,b},$ such that $b\neq0$, the points of $Orb_{G_{1}}[\delta:1:0]$
are not 1-Weierstrass.

\end{lemma}

\begin{proof}
The Hessian $H_F$ at the point $[x:1:0]$ is given by
\begin{equation*}
H_F(x,1,0)=2b(1+x^2)\left(-16a^2x^2+\left(2a+12x^2\right)\left(12+2ax^2\right)\right),
\end{equation*}
so the resultant with $F(x,1,0)$ is given as
\begin{equation*}
Res\left(H_F(x,1,0),F(x,1,0),x\right)=constant.\left(g(a,b)\right)^2,
\end{equation*}
where
\begin{equation*}
g(a,b)=(a-2)^3(a+2)^2b^2\neq0,
\end{equation*}
we have done.
\end{proof}

\begin{lemma}
Let \[B:=Orb_{G_{1}}[\alpha_{1}:\alpha_{1}:1]\cup Orb_{G_{1}}[\alpha_{3}%
:\alpha_{3}:1],\]
then for the quartics $C_{a,b}$ such that $b\neq0,$ we have:
\[B\cap W_{1}(C_{a,b})\neq\phi\text{ }\\
\text{
\ \    }%
\mathrm{if,\,\,and\,\,only\,\,if,\,}\text{ }\,\,\,36-12a+a^{2}-10b^{2}+3ab^{2}=0.\]
Moreover,
\[B\cap W_{1}(C_{a,b})=
\left\{
\begin{array}
[c]{l}%
Orb_{G_{1}}[\alpha_{1}:\alpha_{1}:1],
\text{
\ \ \ \ \ \ \ \ \ \ \ \ \ \ \ }%
\text{if}\text{ }a=\dfrac{1}{2}(12-3b^{2}-\sqrt{9b^{2}-32})\\\\
Orb_{G_{1}}[\alpha_{3}:\alpha_{3}:1],
\text{
\ \ \ \ \ \ \ \ \ \ \ \ \ \ \ }%
\text{if}\text{ }a=\dfrac{1}{2}(12-3b^{2}+\sqrt{9b^{2}-32})\\\\
\phi\text{\,\,\,\,\,\,\,\,\,\,\,\,\,\,\,\,\,\,\,\,\,\,\,\,\,\,\,\,\,\,\,\,\,\,\,\,\,\,\,\,\,\,\,\,\,\,\,\,\,\,\,\,\,\,\,\,\,\,\,\,\,\,\,\,\,\,\,\,\,\,\,\,\,\,\,\, }%
\text{otherwise,}\text{ }
\end{array}
\right.
\]
where
\[
\alpha_{1}=\dfrac{1}{4}\sqrt{\sqrt{9b^{2}-32}-3b}\,\,\,\textrm{and}\,\,\,\alpha_{3}=\dfrac{i}{4}\sqrt{\sqrt{9b^{2}-32}+3b}.
\]
\end{lemma}

\begin{proof}
The resultant of $H_F(x,x,1)$ and $F(x,x,1)$ with respect to $x$ is given by
\begin{equation*}
Res\left(H_F(x,x,1),F(x,x,1),x\right)=constant.\left(g(a,b)\right)^2,
\end{equation*}
where
\begin{equation*}
g(a,b)=\left(a+2\right)\left(2+a-b^2\right)^2\left(36-12a+a^2-10b^2+3a^2\right).
\end{equation*}
Thus%
\begin{equation*}
B\cap W_1(C_{a,b})\neq\phi\,\,\, \text{if,\,\,and\,\,only\,\,if,\,}\,\,\, 36-12a+a^2-10b^2+3ab^2=0.
\end{equation*}
Moreover, we have
\begin{equation*}
F(x,x,1)=(2+a)x^4+2bx^2+1,
\end{equation*}
\begin{equation*}
H_F(x,x,1)=-64b^2x^2\big(b-(-6+a)x^2\big)+16(3+bx^2)\big(-4a^2x^4+\big(%
b+(6+a)x^2\big)^2\big).
\end{equation*}
Thus, substituting $a=\dfrac{1}{2}(12-3b^2\pm\sqrt{9b^2-32})$ in the above system and solving, the required result is obtained.
\end{proof}

\noindent\textbf{Notation.} Let $\Gamma$ be the set of common zeros of the equations
\[
a^{2}+b^{2}-ab^{2}=0\quad \it{and}\quad36-12a+a^{2}-10b^{2}+3ab^{2}=0.
\]

Now the next result is an immediate consequence of the previous
three lemmas.
\begin{proposition}
If $b\neq0$, then for the quartics $C_{a,b},$ we have the following two cases:\\
\vspace*{-.32cm}\textbf{$I.$} $\mathrm{If}\text{ }(a,b)\in\Gamma$,
then
\[X(C_{a,b})\cap W_{1}(C_{a,b})= Orb_{G_{1}}[\alpha:\alpha:1]\cup
Orb_{G_{1}}[\beta_{0}:0:1],\] where
$\alpha\in\{\alpha_{1},\alpha_{3}\},\,\,\beta_{0}\in\{\beta,\dfrac{1}{\beta
}\}.$\\\\
\textbf{$II.$} $\mathrm{If}\text{ }(a,b)\notin\Gamma$, then
\[
X(C_{a,b})\cap W_{1}(C_{a,b})=\left\{
\begin{array}
[c]{l}%
Orb_{G_{1}}[\beta:0:1],\text{
\ \ \ \ \ \ \ \ \ \ \ \ \ \ }%
\text{if}\text{ }a=\dfrac{1}{2}(b^{2}-b\sqrt{b^{2}%
-4})\\\\
Orb_{G_{1}}[\dfrac{1}{\beta}:0:1],\text{
\ \ \ \ \ \ \ \ \ \ \ \ \ }%
\text{if}\text{ } a=\dfrac{1}{2}(b^{2}%
+b\sqrt{b^{2}-4})\\\\
Orb_{G_{1}}[\alpha_{1}:\alpha_{1}:1],\text{
\ \ \ \ \ \ \ \ \ \ \ }%
\text{if}\text{ } a=\dfrac{1}{2}%
(12-3b^{2}-\sqrt{9b^{2}-32})\\\\
Orb_{G_{1}}[\alpha_{3}:\alpha_{3}:1],\text{
\ \ \ \ \ \ \ \ \ \ \ }%
\text{if}\text{ } a=\dfrac{1}{2}%
(12-3b^{2}+\sqrt{9b^{2}-32})\\\\
\phi,\,\,\,\quad\quad\quad\quad\quad\quad\quad\,\,\,\,\,\,\,\,\,\quad\quad \text{otherwise}\text{ }.
\end{array}
\right.
\]
\end{proposition}

\begin{lemma}
If $X(C_{a,b})\cap W_{1}(C_{a,b})\neq\phi$, then the intersection points are necessarily hyperflex points.
\end{lemma}

Now, our main result for the case $b\neq0$ is the following.
\begin{theorem}
For the quartics $C_{a,b}$ such that $b\neq0,$ the number of the 1-Weierstrass points together with their geometry are given in the following tables.
\begin{center}
   \textit{Number Classification on $C_{a,b}$}%
\end{center}
\[%
\begin{tabular}
[c]{|c||c|c|}\hline
& \emph{Ordinary flexes} & \emph{Hyperflexes}\\\hline
$(a,b)\in\Gamma$ & $\mathbf{8}$ & $\mathbf{8}$\\\hline
$P(a,b)=0,\,\,\,Q(a,b)\neq0$ & $\mathbf{0}\,\,\,($\em{Resp.}\,\,$\mathbf{16})$ & $\mathbf{12}\,\,\,($\em{Resp.}\,\,$\mathbf{4})$\\\hline
$P(a,b)\neq0,\,\,\,Q(a,b)=0$ & $\mathbf{0}\,\,\,($\em{Resp.}\,\,$\mathbf{16})$ & $\mathbf{12}\,\,\,($\em{Resp.}\,\,$\mathbf{4})$\\\hline
$P(a,b)\neq0,\,\,\,Q(a,b)\neq0$ & $\mathbf{24}\,\,\,($\em{Resp.}\,\,$\mathbf{8})$ & $\mathbf{0}\,\,\,($\em{Resp.}\,\,$\mathbf{8})$\\\hline
\end{tabular}
\]

\bigskip
\begin{center}
   \textit{Orbit Classification on $C_{a,b}$}%
\end{center}
\[%
\begin{tabular}
[c]{|c|c|c|}\hline
 & Ordinary flexes & Hyperflexes\\\hline
$(a,b)\in\Gamma$ & $1_{8}$ & $2_{4}$\\\hline
$P(a,b)=0,\,\,\,Q(a,b)\neq0$ & \em{None} & $1_{4}$, $1_{8}$\\
& $($\em{Resp.}\,\,$2_{8})$ & $($\em{Resp.}\,\,$1_{4})$\\\hline
$P(a,b)\neq0,\,\,\,Q(a,b)=0$ & \em{None} & $1_{4}$, $1_{8}$\\
& $($\em{Resp.}\,\,$2_{8})$ & $($\em{Resp.}\,\,$1_{4})$\\\hline
$P(a,b)Q(a,b)\neq0$ & $3_{8}$ &\em{None}\\
& $($\em{Resp.}\,\,$1_{8})$ & $($\em{Resp.}\,\,$1_{8})$\\\hline
\end{tabular}
\]

where
\[
P(a,b):=a^{2}+b^{2}-ab^{2},\text{ \ \ \ \ }Q(a,b):=36-12a+a^{2}-10b^{2}%
+3ab^{2}.
\]
\end{theorem}

\begin{proof}
Recall that the number of the 1-Weierstrass points is 24 counted with their weights. Now, if $(a,b)\in \Gamma,$ then by \emph{Proposition 2.13 $(II)$} and \emph{Lemma 2.14}, the intersections $X(C_{a,b})\cap W_{1}(C_{a,b})$ consists of $2_4$ of hyperflex points. Consequently, by \emph{Remark 2.8}, there exist $1_8$ of ordinary flexes.\\
If $P(a,b)=0,\,\,\,Q(a,b)\neq0$ or $P(a,b)\neq0,\,\,\,Q(a,b)=0$, then by \emph{Proposition 2.13 $(I)$} and \emph{Lemma 2.14}, the intersections $X(C_{a,b})\cap W_{1}(C_{a,b})$ consists of $1_4$ of hyperflex points. Consequently, by \emph{Remark 2.8}, there exist either $1_8$ of hyperflex points or $2_8$ of ordinary flexes.\\
If $P(a,b)Q(a,b)\neq0$, then by \emph{Proposition 2.13 $(I)$}, $X(C_{a,b})\cap W_{1}(C_{a,b})=\phi$. Hence, $W_1\big(C_{a,b}\big)$ consists of orbits each of 8 points. Consequently, we have two cases: either $3_8$ of ordinary flexes or $1_8$ of ordinary flexes together with $1_8$ of hyperflex points.\\
This completes the proof.
\end{proof}

\section{Examples}
This section is devoted to illustrate, through examples, the cases
mentioned in \emph{Theorem 2.14}. It should be noted that under a
given  condition more than one case arise, so it is convenient to
investigate whether these cases can occur. For instance,
\begin{itemize}
    \item For the case $P(a,b)=0$ and $Q(a,b)\neq0$, we give the following
    examples:\begin{description}
    \item[(1)] If $a=\dfrac{6}{5},\,\,b=\dfrac{6}{\sqrt{5}}$, then we have
\[
W_{1}(C_{a,b})=Orb_{G_{1}}\left[1.495...i:0:1\right]\cup Orb_{G_{1}%
}[\zeta_{5}:\epsilon_{5}:1],
\]
where
\[
\zeta_{5}:=0.748...(1+i),\,\,\,\epsilon_{5}:=0.748...(1-i).
\]
That is, $C_{a,b}$ has $12$ hyperflex points classified into $1_4$
and $1_8.$ \item[(2)] If $a=3,\,\,b=\dfrac{3}{\sqrt{2}},$ then we
have
\[
W_{1}(C_{a,b})=Orb_{G_{1}}\left[0.841...i:0:1\right]\cup
Orb_{G_{1}}[\zeta_{3}:\epsilon_{3}:1]\cup Orb_{G_{1}}[\zeta
_{4}:\epsilon_{4}:1],
\]
where
\begin{eqnarray*}
\zeta_{3}&:=&0.089...-0.454...i,\,\,\,\epsilon_{3}:=0.188...+0.947...i,\\
\zeta_{4}&:=&0.089...+0.454...i,\,\,\,\epsilon_{4}:=0.188...-0.947...i.
\end{eqnarray*}
Hence, $C_{a,b}$ has $20$ flex points classified into $1_4$ of
hyperflex points and $2_8$ of ordinary flexes.
\end{description}
\end{itemize}
\newpage
\begin{itemize}
    \item For the case $P(a,b)\neq0$ and $Q(a,b)=0$, we get the following two
examples:
\begin{description}
\item[(3)] If $a=3,\,\,b=3$, then $C_{a,b}$ has $12$ hyperflex points
\cite{pa3, pa12}.
\item[(4)] If $a=0,\,\,b=3\sqrt{\dfrac{2}{5}},$ then we have
\[
W_{1}(C_{a,b})=Orb_{G_{1}}\left[0.562...i:0.562...i:1\right]\cup Orb_{G_{1}}[\zeta_{1}:\epsilon_{1}:1]\cup
Orb_{G_{1}}[\zeta_{2}:\epsilon_{2}:1],
\]
where
\begin{eqnarray*}
\zeta_{1}&:=&0.204...-1.151...i,\,\,\,\epsilon_{1}=0.302...-0.269...i,\\
\zeta_{2}&:=&0.204...+1.151...i,\,\,\,\epsilon_{2}=0.302...+0.269...i.
\end{eqnarray*}
Consequently, $C_{a,b}$ has $20$ flex points classified into $1_4$ of hyperflex points and $2_8$ of ordinary flexes.
\end{description}
\end{itemize}

\begin{itemize}
    \item For the case $P(a,b)Q(a,b)\neq0$, we have the following examples:
\begin{description}
\item[(5)] If $a=4,\,\,b=4$, then $C_{a,b}$ has $24$ ordinary flex points
    \cite{pa3, pa12}.
\item[(6)] If $a=-5,\,\,b=1$, then we have\\
\[
W_{1}(C_{a,b})=Orb_{G_{1}}\left[0.44...(1+i):0.44...(i-1):1\right]\cup Orb_{G_{1}}%
\left[0.93...i:2.27...i:1\right]\] Consequently, $C_{a,b}$ has $16$
flex points classified into $2_8$, one consists of hyperflex points
and the other consists of ordinary flexes.
\end{description}
\end{itemize}
\bigskip
{\fontsize{13.5}{9}\textbf{Concluding remarks.}}\\
We conclude the paper by the following remarks and comments.
\begin{itemize}
\item The computations included in the present paper have been performed by the use of MATHEMATICA program. The source code files are available.
  \item The classification of the 1-Weierstrass points of \emph{Kuribayashi quartics with one parameter,} treated in \cite{pa3, pa12} is a particular case of our main theorems. In fact, letting $a=b$, one gets the following table.

       \begin{center}
        Number classification of 1-Weierstrass points of $C_{a,a}$
      \end{center}
        \[
        \begin{tabular}{|c|c|c|}
      \hline
      & Ordinary flex & Hyperflex \\\hline
       $a=0,3$ & $0$ & $12$ \\\hline
        Otherwise & $24$ & $0$ \\
         \hline
         \end{tabular}
      \]
\item The geometry of the 1-Weierstrass points of the one parameter quartic family defined by the equation
 \[
x^{4}+y^{4}+z^{4}+b(x^{2}+y^{2})z^{2}=0,\quad\quad
(b^2-4)(b^2-2)\neq0,%
\]
is a particular case by substituting $a=0$ in \emph{Theorem 2.6} and
\emph{Theorem 2.14} to get the following table.
\begin{center}
  \textit{Number and Orbit Classification on $C_{a,0}$}%
\end{center}
\[%
\begin{tabular}
[c]{|c|c|c|}\hline & \emph{Ordinary flexes} &
\emph{Hyperflexes}\\\hline
$\ $ & $\ $ & $\mathbf{12}$\\
$b=0$ & $\mathbf{0}$ & $1_{4}$\\
&  & $1_{8}$\\\hline
$b^2=\frac{18}{5}$ & $\mathbf{16}$ & $\mathbf{4}$\\
& $2_{8}$ & $1_{4}$\\\hline
Otherwise & $\mathbf{24}\,\,($Resp.\,\,$\mathbf{8})$ & $\mathbf{0}\,\,($Resp.\,\,$\mathbf{8})$\\
& $3_8\,\,($Resp.\,\,$1_{8})$ & None\,\,$($Resp.\,\,$1_{8})$\\\hline
\end{tabular}
\]
\\
The following examples illustrates that under the condition
$b^2\neq0,\,\frac{18}{5}$, both cases mentioned above can occur. For
instance, let $b=\sqrt{6}$, then $C_{0,b}$ has 16 flex points or,
equivalently,
\[W_1\big(C_{0,b}\big)=Orb_{G_1}\left[-0.6...(1+i):0.6...(1-i):1\right]\cup Orb_{G_1}\left[0.2...i:0.2...i:1\right].\]
On the other hand, let $b=\sqrt{5}$, then $C_{0,b}$ has 24 flex
points or, equivalently,
\[W_1\big(C_{0,b}\big)=Orb_{G_1}[\nu_1:\lambda_1:1]\,\,\cup\,\,Orb_{G_1}[\nu_2:\lambda_2:1]\,\,\cup\,\,Orb_{G_1}[\nu_3:\lambda_3:1],\]
where
\begin{eqnarray*}
\nu_1:&=&0.524132...+0.316563...i,\,\,\,\lambda_1:=-0.417502...+0.812472... i,\\
\nu_2:&=&0.524132...-0.316563...i,\,\,\,\lambda_2:=0.417502...+0.812472...i,\\
\nu_3:&=&-0.410813...i,\,\,\,\lambda_3:=-1.37547...i.
\end{eqnarray*}
\item The main theorems constitute a motivation to solve more general problems. One of these problems is the investigation of the geometry of the 1-Weierstrass points of \emph{Kuribayashi quartics with three parameters family} defined by the equation
      \[C_{a,b,c}: x^4+y^4+z^4+ax^2y^2+bx^2z^2+cy^2z^2=0.\]
      Hayakawa \cite{pa2} studied the conditions under which the number of the Weierstrass points of this family is exactly 12 or $<24.$ However, this problem will be the object of a forthcoming work.\\
      Another problem is the investigation of the geometry of higher order and multiple Weierstrass points of $C_{a,b}$ to generalize the classification of the 2-Weierstrass points of \emph{Kuribayashi quartics with one parameter family} \cite{pa3, pa4}.

\item The technique used in this paper is completely different from that followed by Hayakawa \cite{pa5}. Our technique consists of dividing the quartics by group actions into finite orbits and investigate the geometry of these orbits. The results obtained are more informative than those obtained by Hayakawa.
\end{itemize}
\textbf{Acknowledgment} The authors would like express their sincere
gratitude to Prof. Nabil L. Youssef for his guidance throughout the
preparation of this work.


\bigskip\noindent


\begin{thebibliography}{9}                                                                                                %
\bibitem {pa3}Alwaleed K., \emph{Geometry of 2-Weierstrass points on certain
plane curves}, Ph.\,D. Thesis, Saitama Univ., 2010.

\bibitem {pa4}Alwaleed K. and Kawasaki M., \emph{2-Weierstrass points of
certain plane curves of genus three}, Saitama Math. J., \textbf{26} (2009), 49-65.

\bibitem {pa1}Cox, D., Little, J. and O$^{'}$Shea, D., \emph{Ideals, Varities and Algorithms,} Springer-Verlag, New York, 1992.

\bibitem {pa6}Del Centina, A., \emph{Weierstrass points and their impact in
the study of algebraic curves: a historical account from the \textquotedblleft Luckensatz\textquotedblright\,\, to
the 1970s}, Ann Univ. Ferrara, \textbf{54} (2008), 37-59.

\bibitem {pa7}Farkas, H.\,M. and Kra, I., \emph{Riemann Surfaces}, GTM\textbf{71}, Springer Verlag, New York, 1980.

\bibitem {pa23}Francesc B., \emph{Automorphisms groups of genus 3 curves}, Notes del Seminari de Teoria Nombres UB-UAB-UPC 2004/05: Genus 3 curves. Barcelona, Gener 2005.

\bibitem {pa8}Gibson, C.\,G., \emph{Elementary geometry of algebraic curves}, Cambridge University Press, 1998.

\bibitem {pa2}Hayakawa K., \emph{On the family of Riemann surfaces of genus 3 defined by $x^4+y^4+z^4+2tx^2y^2+ux^2z^2+\upsilon y^2z^2=0$}, Bull. Fac. Sci. Eng., Chuo Univ., \textbf{42} (1999), 11-26.

\bibitem {pa5}Hayakawa K., \emph{A copmutation of the numbers of Weierstrass points of a family of Riemann surfaces defined by $x^4+y^4+z^4+2ax^2y^2+2bx^2z^2+2by^2z^2=0$}, Bull. Fac. Sci. Eng., Chuo Univ., \textbf{43} (2000), 1-6.


\bibitem {pa11}Kuribayashi, A. and Komiya, K., \emph{On Weierstrass points of non-hyperelliptic compact Riemann surfaces of genus three}, Hiroshima Math. J., \textbf{7} (1977), 743-768.

\bibitem {pa15}Kuribayashi, A. and Komiya, K., \emph{On Weierstrass points and automorphisms of curves of genus three}, In: \textquotedblleft Algebraic geometry\textquotedblright, Springer, 1979, pp. 253-299.

\bibitem {pa12}Kuribayashi, A. and Sekita, E., \emph{On a family of Riemann
surfaces I}, Bull. Fac. Sci. Eng., Chuo Univ., \textbf{22} (1979), 107-129.


\bibitem {pa13}Miranda, R., \emph{Algebraic Curves and Riemann Surfaces}, Amer.
Math. Soc., 1995.

\bibitem {pa14}Vermeulen, A., \emph{\ Weierstrass points of weight two on
curves of genus three}, Ph.\,D. Thesis, Amsterdam Univ., 1983.
\end{thebibliography}
\end{document}